\documentclass[12pt]{amsart} \usepackage{amsmath,bm,enumitem,graphicx}
\usepackage{amssymb} \usepackage{amsthm}
\usepackage[hidelinks]{hyperref}

\newcommand{\tpdf}{\texorpdfstring}

\newcommand{\Cay}{\operatorname{Cay}}

\newcommand{\mb}{\mathbf}
\newcommand{\mbhalf}{\tfrac{\mb 1}{\mb 2}}
\newcommand{\ssquare}{\scalebox{0.6}{$\square$}}
\newtheorem{theorem}{Theorem}
\newtheorem{lemma}[theorem]{Lemma}

\newtheorem{corollary}[theorem]{Corollary}

\newtheorem{conjecture}[theorem]{Conjecture}
\numberwithin{theorem}{section}
\numberwithin{equation}{section}

\theoremstyle{definition}

\newtheorem{definition}[theorem]{Definition}

\newtheorem{remark}[theorem]{Remark}

\title[Extension of Kriz's example]{Separating topological recurrence from measurable recurrence: exposition and extension of Kriz's example }
\date{\today}

\author{John T. Griesmer}
\email{jtgriesmer@gmail.com}
\address{Department of Applied Mathematics and Statistics, Colorado School of Mines, Golden, Colorado, USA}
\usepackage{amstext}
\begin{document}

\begin{abstract}
We prove that for every infinite set $E\subseteq \mathbb Z$, there is a set $S\subseteq E-E$ which is a set of topological recurrence and not a set of measurable recurrence.   This extends a result of Igor Kriz, proving that there is a set of topological recurrence which is not a set of measurable recurrence.  Our construction follows Kriz's closely, and this paper can be considered an exposition of the original argument.

\end{abstract}
\maketitle

\section{Structure of difference sets}

We write $\mathbb Z$ for the group of integers.  If $A\subseteq \mathbb Z$, we write $A-A$ for the \emph{difference set}, defined to be $\{a_1-a_2:a_1, a_2\in A\}$. If $A, B\subseteq \mathbb Z$, we write $A+B$ for the \emph{sumset}, $\{a+b:a\in A, b\in B\}$. The \emph{upper asymptotic density} (or just \emph{upper density}) of $A$ is $\bar{d}(A):=\limsup_{n\to \infty} \frac{|A\cap \{1,\dots, n\}|}{n}$.  If the limit exists, we write $d(A)$ in place of $\bar{d}(A)$.

\subsection{Kriz's theorem}
We say that $S\subseteq \mathbb Z$ is a \emph{set of density recurrence} (or ``$S$ is \emph{density recurrent}'') if for all $A\subseteq \mathbb Z$, with $\bar{d}(A)>0$, there exists $a, b\in A$ such that $b-a\in S$.  In other words, $(A-A)\cap S\neq \varnothing$ whenever $\bar{d}(A)>0$. The  set $\{n^2 : n\in \mathbb N\}$ of perfect squares is density recurrent, as proved by Furstenberg (\cite{Furstenberg77}, via ergodic theory) and S\'ark\H{o}zy (\cite{Sarkozy}, using the circle method).   This result, along with earlier work by  Bogoliouboff \cite{Bogoliouboff} and F{\o}lner \cite{Folner}, motivated further investigation into the structure of difference sets.

Another result influencing this investigation is van der Waerden's theorem on arithmetic progressions \cite{vdW}: if $\mathbb Z$ is partitioned into finitely many sets $A_1,\dots, A_r$, then at least one of the sets $A_j$ contains arithmetic progressions of every finite length.  The ``density version'' of van der Waerden's theorem is Szemer\'edi's theorem \cite{Sz}, which says that if $\bar{d}(A)>0$, then $A$ contains arithmetic progressions of every finite length.  These results suggested a possible deep connection between partition results and density results.  To formalize this suggestion, we say a set $S\subseteq \mathbb Z$ is a \emph{set of chromatic recurrence} (or ``is \emph{chromatically recurrent}'') if for every partition $\{A_1,\dots, A_r\}$ of $\mathbb Z$ into finitely many sets, there is a cell $A_j$ of the partition such that $(A_j-A_j)\cap S\neq \varnothing$. Upper density is finitely subadditive, so if $\mathbb Z$ is partitioned into finitely many sets, then at least one of those sets has positive upper density, and its difference set therefore intersects every density recurrent set.  Thus every density recurrent set is also chromatically recurrent.  Bergelson \cite{BergelsonERT},  Furstenberg \cite{FurstenbergBook}, and Ruzsa \cite{Ruzsa82} each asked whether the converse holds: is every chromatically recurrent set also density recurrent?  Igor Kriz gave a negative answer in the following theorem from \cite{Kriz}.  The main purpose of this article is to give an expository proof.
\begin{theorem}\label{thm:Kriz}
  There is a set of integers which is chromatically recurrent and not density recurrent.
\end{theorem}

Readers familiar with \cite{Kriz} will see that the methods there can be combined with Lemma \ref{lem:InfiniteDense} below to prove the following generalization of Theorem \ref{thm:Kriz}.   Our second purpose in this article is to give an explicit proof of this generalization.
\begin{theorem}\label{thm:KrizInSminusS}
  If $E$ is an infinite set of integers, then there is a subset $S\subseteq E-E$ such that $S$ is chromatically recurrent and not density recurrent.
\end{theorem}

\begin{remark} In terms of dynamical systems, Theorem \ref{thm:Kriz} may be stated as ``there is a set of topological recurrence which is not a set of measurable recurrence.''  Likewise,  Theorem \ref{thm:KrizInSminusS} says that if $E\subseteq \mathbb Z$ is infinite, then there is subset of $E-E$ which is a set of topological recurrence and not a set of measurable recurrence; see \cite{BergelsonMcCutcheon} for a general discussion of the equivalences between various recurrence properties.  Since our constructions do not explicitly reference dynamical systems, we only use the terms ``chromatic recurrence'' and ``density recurrence.''
\end{remark}

%Our interest in Theorem \ref{thm:KrizInSminusS} is motivated by the following results on difference sets.
%
%\begin{enumerate}
%\item[(i)] If  $S$ is an infinite set of integers, the set $\Delta(S):=\{b-a: a,b\in S, a\neq b\}$  is density recurrent. This is really Poincar\'e's recurrence theorem for measure preserving dynamical systems (see \cite{FurstenbergBook}, \cite{BergelsonMultifarious} for discussion), but the connection to the  arithmetic formulation here was made long after both results had been proved.
%
%\item[(ii)]  \cite[Theorem 2.3]{KunenRudin} If $S\subseteq \mathbb Z$ is lacunary, and $t\notin S-S$, then $(S-S)-t$ is not Bohr recurrent.  This indicates that for such $S$, only ``local'' properties of $S-S$ can determine various recurrence properties of $S-S$.
%\end{enumerate}

  There are several proofs of Theorem \ref{thm:Kriz} (\cite{Ruzsa85},\cite{ForrestThesis},\cite{McCutcheon},\cite{McCutcheonBook},\cite{WeissBook}), all of which follow the same broad outline as \cite{Kriz} and overcome the essential difficulties in a similar way.  Hopefully our exposition will help readers solve some related open problems, or  find a fundamentally different approach to Theorems \ref{thm:Kriz} and \ref{thm:KrizInSminusS}.
  Our approach in this article is mathematically nearly identical to Kriz's original construction in \cite{Kriz}, except that our approach to Lemma \ref{lem:2Pieces} is closer to an argument in Ruzsa's version of the construction \cite{Ruzsa85} (also presented in \cite{McCutcheon, McCutcheonBook}).

\begin{remark}
It may seem unnatural to use upper density when studying subsets of $\mathbb Z$, since $\bar{d}(-\mathbb N)=0$.  Intuitively $\bar{D}(A):=\limsup_{n\to\infty} \frac{|A\cap \{-n,\dots,n\}|}{2n+1}$ is a more natural notion of density for subsets of $\mathbb Z$. However, upper asymptotic density is more convenient for the proof of Lemma \ref{lem:2Pieces}.  This suggests that we should work in $\mathbb N$ rather than in $\mathbb Z$; we work in $\mathbb Z$ because a key step in our proofs will use group homomorphisms from $\mathbb Z$ into other groups.
\end{remark}

\subsection{Outline of the article}\label{sec:outline}

Our proof of Theorems \ref{thm:Kriz} and \ref{thm:KrizInSminusS} consists of two independent steps.

\medskip

\noindent \textbf{Step I.} Identify finite subsets of $\mathbb Z$ which approximate, in a precise sense, the property of being chromatically recurrent and not density recurrent.  This is split into two substeps:

\begin{enumerate}
\item[\textbf{(A)}] Find subsets of a finite group $G$ approximating  the property of being chromatically recurrent and not density recurrent.

\item[\textbf{(B)}] Copy these sets from $G$ into $\mathbb Z$ in a way that maintains their recurrence properties.
\end{enumerate}

\noindent \textbf{Step II.} Piece together the finite sets found in Step I in a way that maintains their recurrence properties.

\medskip

In \S\ref{sec:Model} we prove a version of Theorem \ref{thm:Kriz} for groups of the form $(\mathbb Z/2\mathbb Z)^d$. This carries out Step I(A) of the outline and introduces one of the key ideas in a clearer setting.

In \S\ref{sec:Proofs} we state and prove Lemma \ref{lem:2Pieces}, completing Step II of the outline.  We then state the other main lemmas, Lemmas \ref{lem:FinitePiecesExist} and \ref{lem:PiecesInSminusS}, and prove Theorems \ref{thm:Kriz} and \ref{thm:KrizInSminusS}.  The latter two lemmas are proved in \S\ref{sec:BH}, carrying out Part I(B) of the outline.

The proofs of Theorems \ref{thm:Kriz}  and \ref{thm:KrizInSminusS} are mostly elementary. The only highly nontrivial result required, for both proofs, is a lower bound for chromatic numbers of Kneser graphs, discussed in \S\ref{sec:CayleyKneser}.  Theorem \ref{thm:KrizInSminusS} uses some standard results on subsets of $\mathbb T^d$ of the form $\{n\bm\alpha:n\in E\}$ for arbitrary infinite sets $E\subseteq \mathbb Z$; we provide proofs in \S\ref{sec:UD} for completeness.  We also use, without proof, the elementary binomial estimate $\lim_{n\to \infty} \binom{n}{\lfloor n/2\rfloor}2^{-n}=0.$

Section \ref{sec:Questions} contains some remarks and open problems.
%
%Subsequent articles will discuss these results as part of the general theory of recurrence in dynamical systems.

%
%(ii) \cite{KunenRudin} If $S\subseteq \mathbb Z$ is lacunary, then for all $t\neq 0$, the set $((S-S)+t)\setminus \{0\}$ is not Bohr recurrent.  This suggests that only ``local'' features of a subset  $S'\subseteq \Delta(S)$ can determine recurrence properties of $S'$.
%
%(iii)  If $S\subseteq \mathbb Z$ is lacunary and $S'\subseteq \Delta S$ is Bohr recurrent, then $S'$ is chromatically recurrent.
%

%
%The following result is an immediate consequence of Theorem \ref{thm:KrizInSminusS}.
%
%\begin{corollary}\label{cor:WeakSminusS}
%  If $S$ is an infinite set of integers, then there is a set $E\subseteq S-S$ such that $E$ is a set of Bohr recurrence and not a set of measurable recurrence.
%\end{corollary}
%A strange feature of our present state of knowledge: the only known way of proving Corollary \ref{cor:WeakSminusS} uses lower bounds for chromatic numbers of Kneser graphs, in contrast to Corollary \ref{cor:WeakKriz}, which can be proved by more elementary means.

\section{A model setting}\label{sec:Model}
While Theorems \ref{thm:Kriz} and \ref{thm:KrizInSminusS} are about sets of integers, one of our main arguments is easier to develop in finite vector spaces over the field with two elements. These results in this section will be used as building blocks for the main construction in \S\S\ref{sec:Proofs}-\ref{sec:CopyCayley}. Forrest takes a similar approach in \cite{ForrestThesis} and in \cite{Forrest}.  %This section (excepting \S\ref{sec:CayleyKneser}) can be skipped if one only wants to read a proof of Theorems \ref{thm:Kriz} and \ref{thm:KrizInSminusS}, but it will be easier to follow the material of \S\ref{sec:BH} after reading this section.

\subsection{Finite approximations to recurrence}

Let $\Gamma$ be a finite abelian group and $S\subseteq \Gamma$.  We say that $S$ is
\begin{enumerate}
    \item[$\bullet$] \emph{$\delta$-density recurrent} if $(A-A)\cap S\neq \varnothing$ for all $A\subseteq \Gamma$ having $|A|>\delta |\Gamma|$;
  \item[$\bullet$] \emph{$\delta$-nonrecurrent} if there exists $A\subseteq \Gamma$ having $|A|>\delta |\Gamma|$ such that $(A-A)\cap S=\varnothing$;
  \item[$\bullet$] \emph{$k$-chromatically recurrent} if for every partition of $\Gamma$ into $k$ sets $A_1, A_2,\dots, A_k$, the intersection $(A_j-A_j)\cap S$ is nonempty for some $j\leq k$.
\end{enumerate}
Note that if $0_\Gamma\in S$, then $S$ is both $k$-chromatically recurrent for every $k$ and  $\delta$-density recurrent for all $\delta\geq 0$, as $0_\Gamma \in A-A$ for every nonempty $A$.  However, our interest is only in sets $S$ not containing $0_\Gamma$.

If $d\in \mathbb N$, we write $\mathbb F_2^d$ for the group $(\mathbb Z/2\mathbb Z)^d$, the product of $d$ copies of $\mathbb Z/2\mathbb Z$.  Elements of $\mathbb F_2^d$ will be written as $\mb x = (x_1,\dots,x_d)$, where $x_j\in \{0,1\}$.  The remainder of this section is dedicated to proving the following analogue of Theorem \ref{thm:Kriz}.  %We will use the material from \S\ref{sec:CayleyKneser} in the main sections, but the remaining portions of this section will not be used in our main arguments.
\begin{theorem}\label{thm:KrizF2}
  Let $\delta<\frac{1}{2}$ and $k\in \mathbb N$.  For all sufficiently large $d$, there exists $S\subseteq \mathbb F_2^d$ such that $S$ is $k$-chromatically recurrent and $\delta$-nonrecurrent.
\end{theorem}

The next subsection summarizes background for the proof of Theorem \ref{thm:KrizF2}.

\subsection{Cayley graphs and Kneser graphs}\label{sec:CayleyKneser} We adopt the usual terminology from graph theory:  a \emph{graph} $\mathcal G$ is a set $V$, whose elements are called \emph{vertices}, together with a set $E$ of unordered pairs of elements of $V$, called \emph{edges}.  A \emph{vertex coloring of} $\mathcal G$ with $k$ colors (briefly, a $k$-coloring) is a function $f:V\to \{1,\dots,k\}$.  We say that $f$ is \emph{proper} if $f(v_1)\neq f(v_2)$ for every edge $\{v_1,v_2\}\in E$.  The \emph{chromatic number} of $\mathcal G$ is the smallest $k$ such that there is a proper $k$-coloring of $G$.

Let $r\leq n\in \mathbb N$. The \emph{Kneser graph} $KG(n,r)$ is the graph whose vertices are the $r$-element subsets of $\{1,\dots,n\}$, with two vertices $A, B\subseteq \{1,\dots,n\}$  joined by an edge if and only if $A\cap B=\varnothing$.  Note that $KG(n,r)$ has no edges if $2r> n$.  M.~Kneser proved that the chromatic number of $KG(n,r)$ is no greater than $n-2r+2$, and   conjectured that this is the correct value when $2r\leq n$. Lov\'asz proved Kneser's conjecture in \cite{Lovasz}, and B\'{a}r\'{a}ny \cite{Barany} provided a more elementary proof.  An even shorter proof was provided by Greene in \cite{GreeneKneserShort}. Matousek's book \cite{Matousek} gives a detailed exposition of these proofs and subsequent developments of the techniques they introduced.
\begin{theorem}[\cite{Lovasz}]\label{thm:Lovasz}
  If $2r\leq n\in \mathbb N$, then the chromatic number of $KG(n,r)$ is $n-2r+2$.
\end{theorem}
Examination of our proofs shows that we need only the following corollary of Theorem \ref{thm:Lovasz}.
\begin{corollary}\label{cor:ChromaticLowerBound}
  If $n_k\to\infty $, $r_k\to \infty$, and $r_k^2/n_k\to 0$, then the chromatic number of $KG(2n_k+r_k,n_k)$ tends to $\infty$.
\end{corollary}
Surprisingly, there are no known proofs of nontrivial lower bounds for the chromatic number of $KG(n,r)$, besides those that prove Theorem \ref{thm:Lovasz}.  A substantially different proof of Corollary \ref{cor:ChromaticLowerBound} could lead to progress on problems in recurrence (cf. Remark \ref{rem:AvoidKneser}).

Given an abelian group $\Gamma$ and a subset $S\subseteq \Gamma$, the \emph{Cayley graph based on $S$}, denoted $\Cay(S)$, is the graph whose vertex set is $\Gamma$, with two vertices $x,y$ joined by an edge if $x-y\in S$ or $y-x\in S$.  It follows immediately from the definitions that $S$ is $k$-chromatically recurrent if and only if the chromatic number of $\Cay(S)$ is strictly greater than $k$.  To prove $S$ is $k$-chromatically recurrent, it therefore suffices to prove that the chromatic number of $\Cay(S)$ is at least $k+1$.

Note that a Cayley graph contains loops (i.e.~edges of the form $\{x\}$) if and only if $0_\Gamma\in S$.  As mentioned at the beginning of this section, we are only interested in recurrence properties of sets $S$ not containing $0_\Gamma$, so the Cayley graphs we consider will have no loops.

\subsection{Hamming balls in \texorpdfstring{$\mathbb F_2^d$}{F2d}}
 For $\mb x=(x_1,\dots,x_d)\in \mathbb F_2^d$, we define
\[
w(\mb x) := |\{j\leq d: x_j\neq 0\}|.
\]
%It follows immediately from the definition that
%\begin{equation}\label{eqn:HammingTriangle}
%w(\mb x+\mb y)\leq w(\mb x)+w(\mb y) \quad \text{for all } \mb x, \mb y\in \mathbb F_2^d.
%\end{equation}
Given $k\leq d$ and $\mb y\in \mathbb F_2^d$, the \emph{Hamming ball of radius $k$ around $\mb y$} is
\[
H_k(\mb y):=\{\mb x\in \mathbb F_2^d:w(\mb y-\mb x)\leq k\}.
\]
So $H_k(\mb y)$ is the set of $\mb x = (x_1,\dots, x_d)$ such that $x_j\neq y_j$ for at most $k$ coordinates $j$.

Let $\mb 0=(0,\dots, 0)$ and $\mb 1=(1,\dots,1)\in \mathbb F_2^d$.  The following identities are easy to verify from the definitions:
\begin{equation}\label{eqn:HkHr}
  H_k(\mb 0)-H_k(\mb 0) = H_{2k}(\mb 0),
\end{equation}
\begin{equation}\label{eqn:AntipodalH}
  H_{k}(\mb 0)\cap H_{r}(\mb 1) = \varnothing \text{ iff } k + r < d.
\end{equation}

\begin{lemma}\label{lem:HNonrecurrent}
  Let $k\in \mathbb N$ and $\delta<\frac{1}{2}$.  If $d\in \mathbb N$ is sufficiently large, then $H_k(\mb 1)\subseteq \mathbb F_2^d$ is $\delta$-nonrecurrent.
\end{lemma}
We will use the following well known estimates on binomial coefficients:
setting  $M_d:=\max_{0\leq j \leq d} \binom{d}{j}$, we have  $\lim_{d\to\infty} M_d/2^d=0$.  The identities $\binom{d}{j} = \binom{d}{d-j}$ and $\sum_{j=0}^d \binom{d}{j}=2^d$ then imply that for fixed $k\in \mathbb N$ and $\delta<\frac{1}{2}$, we have the following for all large enough $d$:
\begin{equation}\label{eqn:Binom}
\sum_{j=0}^{\lfloor d/2\rfloor-k} \binom{d}{j}>\delta 2^d.
\end{equation}
\begin{proof}[Proof of Lemma \ref{lem:HNonrecurrent}]
  Fix $k\in \mathbb N$ and $\delta< \frac{1}{2}$.  Choose $d$ to be sufficiently large that (\ref{eqn:Binom}) holds.  Let $A=H_{\lfloor d/2\rfloor-k}(\mb 0)$, so that $A$ is the set of elements $(x_1,\dots,x_d)\in \mathbb F_2^d$ having at most $\lfloor d/2\rfloor-k$ entries equal to $1$.  Now $|A|$ is given by the sum in (\ref{eqn:Binom}), so $|A|>\delta |\mathbb F_2^d|$.  Equation (\ref{eqn:HkHr}) implies $A-A=H_{2(\lfloor d/2 \rfloor-k)}(\mb 0)$, which is disjoint from $H_k(\mb 1)$ by (\ref{eqn:AntipodalH}). So we have shown that $H_k(\mb 1)$ is $\delta$-nonrecurrent. \end{proof}

\begin{lemma}\label{lem:HContainsKneser}
  Let $k, d\in \mathbb N$ with $2k\leq d$.  The Cayley graph $\Cay(H_{2k+1}(\mb 1))$ in $\mathbb F_2^d$ contains a subgraph isomorphic to $KG(d,\lfloor d/2\rfloor-k)$.  Consequently, $H_{2k+1}(\mb 1)$ is $2k$-chromatically recurrent.
\end{lemma}

\begin{proof}
Let $\mathcal G$ denote the Cayley graph $\Cay(H_{2k+1}(\mb 1))$ and let $r=\lfloor d/2\rfloor-k$.  To each $C\subseteq \{1,\dots,d\}$ having cardinality $r$, we will assign an element $\mb x_C\in \mathbb F_2^d$ (without repetition), and we will show that if such $C, C'$ are disjoint (meaning they are joined by an edge in $KG(d,r)$), then $\mb x_C-\mb x_{C'} \in  H_{2k+1}(\mb 1)$  (meaning  $\mb x_C$ and $\mb x_{C'}$ are joined by an edge in $\mathcal G$).  The set of $\mb x_C$ so chosen will thereby determine a subgraph of $\Cay(H_{2k+1}(\mb 1))$ isomorphic to $KG(d,r)$.

For each $C\subseteq \{1,\dots,d\}$ having $|C|=r$, we let $\mb x_C:=1_C$.  This is the characteristic function of $C$, viewed as an element of $\mathbb F_2^d$, so $\mb x_C$ has exactly $r$ entries equal to $1$.  Now if $C\cap C'=\varnothing$, then $\mb x_C-\mb x_{C'}$ has exactly $2r$ entries equal to $1$. Since $d-2r\leq 2k+1$, this means $\mb x_C-\mb x_{C'}\in H_{2k+1}(\mb 1)$, so that $\mb x_C$ and $\mb x_{C'}$ are joined by an edge of $\mathcal G$.

According to Theorem \ref{thm:Lovasz}, the chromatic number of $KG(d,r)$ is  $d-2r+2$.  With $r=\lfloor d/2\rfloor-k$, we get $d-2r+2\geq 2k+1$, so $\mathcal G$ has chromatic number at least $2k+1$, which implies that $H_{2k+1}(\mb 1)$ is $2k$-chromatically recurrent.
\end{proof}

\begin{proof}[Proof of Theorem \ref{thm:KrizF2}]
Fix $\delta\in (0,\frac{1}{2})$ and $k\in \mathbb N$.  Combining Lemmas \ref{lem:HNonrecurrent} and \ref{lem:HContainsKneser}, we get that for all sufficiently large $d$, the Hamming ball $H_{2k+1}(\mb 1)\subseteq \mathbb F_2^d$ is $\delta$-nonrecurrent and $k$-chromatically recurrent.
\end{proof}

%
%The next proposition is the main result of \cite{Kleitman}, will be used in the proof of Theorem \ref{thm:ForrestDiff}.
%
%\begin{proposition}\label{prop:Kleitman}
%If $A\subseteq \mathbb F_2^d$ has $|A|>\sum_{j=0}^{k} \binom{d}{j}$, then $(A-A)\cap H_{2k}(\mb 1)\neq \varnothing$.
%\end{proposition}

%Let $A\subseteq \mathbb F_2^d$, $A:=\{\mb x : w(\mb x)\leq \lfloor\frac{d}{2} \rfloor+1\}$ and $\mb t \in H_{2k}(\mb 1)$.  Then
%\[|A\cap (A+\mb t)|\leq \sum_{j =\lfloor d/2\rfloor -2k}^{\lfloor d/2\rfloor + 2k} \binom{d}{j}=o(2^d).\]

\section{Proof of Theorems \ref{thm:Kriz} and \ref{thm:KrizInSminusS}}\label{sec:Proofs}

\subsection{Assembling finite pieces}
We now return to $\mathbb Z$ and prove Theorems \ref{thm:Kriz} and \ref{thm:KrizInSminusS}.  These theorems involve sets which are not density recurrent, but do have some other recurrence property.  Lemma \ref{lem:2Pieces} provides a general procedure for building such sets from finite pieces.  To make this idea precise, we need the following definitions.

Let $S\subseteq \mathbb Z$.  We say that $S$ is
\begin{enumerate}
  \item[$\bullet$] \emph{$\delta$-density recurrent} if $(A-A)\cap S\neq \varnothing$ for all $A\subseteq \mathbb Z$ having $\bar{d}(A)>\delta$.
  \item[$\bullet$] \emph{$\delta$-nonrecurrent} if there exists $A\subseteq \mathbb Z$ having $\bar{d}(A)>\delta$ such that $(A-A)\cap S=\varnothing$.  In other words, $S$ is not $\delta$-density recurrent.
  \item[$\bullet$] \emph{$k$-chromatically recurrent} if for every partition of $\mathbb Z$ into $k$ sets $A_1, A_2,\dots, A_k$, the intersection $(A_j-A_j)\cap S$ is nonempty for some $j\leq k$.

      Equivalently, $S$ is $k$-chromatically recurrent if for every function $f:\mathbb Z\to \{1,\dots,k\}$, there exist $a, b\in \mathbb Z$ such that $f(a)=f(b)$ and $b-a\in S$.
%  \item[$\bullet$] \emph{$k$-Bohr recurrent} if $B\cap S\neq \varnothing$ for every Bohr neighborhood of $0$ having rank at most $k$.
%  \item[$\bullet$] \emph{$(k,\varepsilon)$-Bohr recurrent} if $B\cap S\neq \varnothing$ for every Bohr neighborhood of $0$ having rank at most $k$ and radius greater than $\varepsilon$.
\end{enumerate}

\begin{remark}
  Note that $(A-A)\cap S=\varnothing$ if and only if $A\cap (A+S)=\varnothing$.  This equivalence will be used from time to time without comment.  Likewise we will use the following trivial observations, true for all $A, B\subseteq \mathbb Z$ and $t\in \mathbb Z$:
\begin{enumerate}
  \item[$\bullet$] $(A-t)-(A-t) = A-A$;
  \item[$\bullet$] $(A\cap B) -t = (A-t)\cap (B-t)$.
\end{enumerate}

\end{remark}

If $m\in \mathbb N$, we write $[m]$ for the interval $\{0,1,\dots, m-1\}$ in $\mathbb Z$.

\begin{definition} If $m\in \mathbb N$ and $B\subseteq [ farkm]$, we say that $(B,m)$ \emph{witnesses the $\delta$-nonrecurrence of $S$} if $|B|>\delta m$ and $B\cap (B+S)=\varnothing$,  $B+S\subseteq [m]$, and $B+S+S\subseteq [m]$.
\end{definition}
Note that if there is an $m\in \mathbb N$ and $B\subset [m]$ such that $(B,m)$ witnesses the $\delta$-nonrecurrence of $S$, then $S$ is $\delta$-nonrecurrent: the set $B':=B+m\mathbb Z$ has density $d(B')=\frac{1}{m}|B|>\delta$, and the conditions $B, B+S\subset [m]$ imply $B'\cap (B'+S) = B\cap (B+S) + m\mathbb Z$.  Since $B\cap (B+S)=\varnothing$, we have $B'\cap (B'+S)=\varnothing$.

The condition $B+S+S\subseteq [m]$ may seem unmotivated, but it will be used in the proof of Lemma \ref{lem:2Pieces}.

\begin{lemma}\label{lem:IntersectionDensity}
If $A\subseteq \mathbb N$ and $m\in \mathbb N$, then there is a $t\in \mathbb Z$ such that $|A\cap ([m]+t)|\geq \bar{d}(A)m$.
\end{lemma}

\begin{proof}
Let $m\in \mathbb N$ and let
\[\delta := \sup_{t\in \mathbb Z} \frac{|A\cap([m]+t)|}{m},\] so that $|A\cap ([m]+t)|\leq \delta m$ for every $t\in \mathbb Z$.   Given $N\in \mathbb N$, write $[N]$ as a union of $\lfloor N/m \rfloor$ disjoint intervals $I_1,\dots, I_{\lfloor N/m \rfloor}$ of length $m$, together with another (possibly empty) interval $I_{\lfloor N/m\rfloor+1}$ of length at most $m$. Then $|A\cap I_j|\leq \delta m$ for every $j$, so
\[|A\cap [N]|\leq \sum_{j=1}^{\lfloor N/m\rfloor + 1} |A\cap I_j| \leq  \lfloor N/m \rfloor \delta m  + m \leq \delta N + m.\]  Then $\limsup_{N\to \infty} \frac{|A\cap [N]|}{N} \leq \delta$, meaning $\bar{d}(A)\leq \delta$.  Since there are only finitely many possible values of $\frac{|A\cap ([m]+t)|}{m}$, the supremum is attained.
\end{proof}

\begin{lemma}\label{lem:WitnessesExist}
  If $S\subseteq \mathbb Z$ is finite and $\delta$-nonrecurrent, then for all sufficiently large $m$, there is a set $B\subseteq [m]$ such that $(B,m)$ witnesses the $\delta$-nonrecurrence of $S$.
\end{lemma}

\begin{proof}
  Assuming $\delta>0$ and that $S$ is finite and $\delta$-nonrecurrent, there is  set $A\subseteq \mathbb Z$ such that $\bar{d}(A)>\delta$ and $(A-A)\cap S=\varnothing$.  Fix such an $A$. Let $k=\max\{|n|:n\in S\}$, and choose $m_0$ sufficiently large that $\bar{d}(A)m-2k > \delta m$ whenever $m>m_0$.

   Let $m>m_0$.  We will find a set $B\subseteq [m]$ such that $(B,m)$ witnesses the $\delta$-nonrecurrence of $S$.  By Lemma \ref{lem:IntersectionDensity}, choose $t\in \mathbb Z$ so that $|A\cap ([m]+t)|\geq \bar{d}(A)m$.  Then $|(A-t)\cap [m]|\geq \bar{d}(A)m$.  Let $B= (A-t)\cap [m-2k]$, so that $|B|\geq \bar{d}(A)m-2k$.  Then $|B|> \delta m$, by our choice of $m_0$.  Since $B$ is contained in a translate of $A$, we have $B-B\subseteq A-A$, meaning $B-B$ is disjoint from $S$.  The containment $B\subseteq [m-2k]$ implies $B+S, B+S+S\subseteq [m]$.
\end{proof}

The next lemma  %- \ref{lem:DilateBohrRecurrent}
will be used in conjunction with Lemma \ref{lem:2Pieces} below.  If $S\subseteq \mathbb Z$ and $m\in \mathbb N$,  we write $mS$ for $\{mn: n\in S\}$.

\begin{lemma}\label{lem:DilateChromatic}
  Let $k,m\in \mathbb N$. If $S\subseteq \mathbb Z$ is $k$-chromatically recurrent then so is $mS$.
\end{lemma}

\begin{proof}
  Suppose $m\in \mathbb N$, $S$ is $k$-chromatically recurrent, and let $f:\mathbb Z\to \{1,\dots,k\}$.  Form a new coloring $\tilde{f}:\mathbb Z\to \{1,\dots,k\}$ by $\tilde{f}(n)=f(mn)$.  Since $S$ is $k$-chromatically recurrent, there exists $a,b\in \mathbb Z$ such that $b-a\in S$ and $\tilde{f}(a)=\tilde{f}(b)$.  This means $f(ma)=f(mb)$, and $mb-ma\in mS$.  Since $f$ was an arbitrary function into $\{1,\dots,k\}$, this shows that $mS$ is $k$-chromatically recurrent. \end{proof}

%  Now suppose $E$ is not $k$-chromatically recurrent, and let $f:\mathbb Z\to \{1,\dots,k\}$ be such that $f(a)\neq f(b)$ whenever $b-a\in E$.  Define a new coloring $\tilde{f}:\mathbb Z\to \{1,\dots,k\}$ by $\tilde{f}(mq+r)=f(q)$ whenever $0\leq r \leq m-1$.  Now if $\tilde{f}(a)=\tilde{f}(b)$ with $b-a\in mE$, then $a$ and $b$ are congruent mod $m$, so we can write $a=mq_a+r$ and $b=mq_b+r$, where $0\leq r\leq m-1$.  Now the definition of $\tilde{f}$ implies $f(q_a)=f(q_b)$, and $q_b-q_a\in E$, contradicting our choice of $f$.  We therefore have that $\tilde{f}(a)\neq \tilde{f}(b)$ whenever $b-a\in mE$, meaning $mE$ is not $k$-chromatically recurrent.
%\end{proof}

%\begin{lemma}\label{lem:DilateBohrRecurrent}
%  If $S\subseteq \mathbb Z$ is $(k,\varepsilon)$-Bohr recurrent and $m\in \mathbb N$, then $mS$ is $(k,\varepsilon)$-Bohr recurrent.
%\end{lemma}
%
%\begin{proof}
%  Let $\alpha_1,\dots,\alpha_k\in [0,1)$.  Since $S$ is $(k,\varepsilon)$-Bohr recurrent, we may choose $s\in S$ such that $\|sm\alpha_j\|\leq \varepsilon$ for all $j\leq k$.  Since $sm\in mS$, this shows that $mS$ is $(k,\varepsilon)$-Bohr recurrent.
%\end{proof}

%\subsection{The combining lemma}
Here is the first key lemma for our constructions, a variant of Lemma 3.2 in \cite{Kriz}.  We continue to use $[m]$ to denote $\{0,\dots,m-1\}$.

\begin{lemma}\label{lem:2Pieces}
  Let $\delta,\eta\in (0,\frac{1}{2})$, and let $E, F\subseteq \mathbb N$ be finite sets which are $\delta$-nonrecurrent and $\eta$-nonrecurrent, respectively.  If $(A,m)$ witnesses the $\delta$-nonrecurrence of $E$, then for all sufficiently large $l\in \mathbb N$, there exists $C\subseteq [lm]$ with $A\subseteq C$ such that $(C,lm)$ witnesses the $2\delta\eta$-nonrecurrence of $E\cup mF$.
Consequently, for all sufficiently large $m$, $E\cup mF$ is $2\delta\eta$-nonrecurrent.
\end{lemma}
In practice we will apply Lemma \ref{lem:2Pieces} with $\eta$ close to $\frac{1}{2}$, so that the quantity $2\delta\eta$ in the conclusion will be close to $\delta$ in the hypothesis.

\begin{proof}
Let $E$, $F$, $A$, and $m$ be as in the hypothesis of the lemma.  Since $(A,m)$ witnesses the $\delta$-nonrecurrence of $E$, we have $A\subseteq [m]$ and  $|A|>\delta m$.  Write $|A|$ as $\delta' m$, so that $\delta'>\delta$. Let $k=\max(E\cup mF)$, and choose $l_0$ so that for all $l>l_0$, we have \begin{equation}\label{eqn:LargerDelta}2\delta'\eta lm-2k> 2\delta\eta lm.
\end{equation}
Let $l\in \mathbb N$ be greater than $l_0$ and large enough (by Lemma \ref{lem:WitnessesExist}) that there is a set $B\subseteq [l]$ such that $(B,l)$ witnesses the $\eta$-nonrecurrence of $F$;  fix such a $B$.  We will form $C$ as a union of translates of $A$.  Each such translate $A+t$ will lie in one of the mutually disjoint intervals
\[
I_0=[0, m-1], I_1=[m, 2m-1], \dots, I_{l-1}=[(l-1)m,lm-1],
\] and will be arranged so that if $A+t\subseteq I_j$, then $A+t+E\subseteq I_j$.   First fix arbitrary elements $e_0\in E$ and $f_0\in F$. Let
  \[C_1:= \bigcup_{b\in B} A+mb, \qquad C_2:=\bigcup_{b\in B} A+e_0+m(b+f_0), \qquad C:=(C_1\cup C_2)\cap [lm-2k].\]
 We claim that every element of $C$ can be written uniquely as
 \begin{equation}\label{eqn:represent}
 a+qe_0 + m(b+qf_0), \qquad \text{where } a\in A, b\in B, q\in \{0,1\}.
 \end{equation}  The existence of such a representation is evident from the definition of $C$.   To prove uniqueness, assume that $a,a'\in A$, $b,b'\in B$,  $q,q'\in \{0,1\}$, and
  \begin{equation}
  \label{eqn:TwoSuch} a+qe_0 + m(b+qf_0) = a'+q'e_0 + m(b'+q'f_0),
  \end{equation}
   with the aim of proving $a=a'$, $b=b'$, and $q=q'$. First observe that every element of $[ml]$ can be written uniquely as $s+mt$ with $s\in [m]$ and $t\in [l]$.  Then $a+qe_0, a'+q'e_0\in [m]$ and $b+qf_0, b'+q'f_0\in [l]$, so $a+qe_0=a'+q'e_0$, meaning $a-a' = (q'-q)e_0$.  This implies $a-a'\in  \{0\} \cup \pm E$, so the assumption $(A-A)\cap E = \varnothing$ implies $a=a'$, whence $q=q'$ as well. Equation (\ref{eqn:TwoSuch}) then simplifies to $a+qe_0 + m(b+qf_0) = a+qe_0 + m(b'+qf_0)$, implying $b=b'$.

 We will prove that
 \begin{align}
 \label{eqn:Cdensity} &|C|>2\delta\eta lm, \\
 \label{eqn:CEFinLM}  & C+(E\cup mF)+(E\cup mF)\subseteq [lm], \\
 \label{eqn:CCEFempty} & C\cap \bigl(C+(E\cup mF)\bigr)=\varnothing.
 \end{align}
The unique representation of elements of $C$ in (\ref{eqn:represent}) implies $|C|\geq 2|A||B|-2k$ (the subtraction accounts for the containment in $[lm-2k]$).   Our choice of $A$ and $B$ together with (\ref{eqn:LargerDelta}) implies $2|A||B|-2k >2\delta'\eta lm-2k    > 2\delta\eta lm$; combined with the lower bound on $|C|$ this proves (\ref{eqn:Cdensity}). The containment (\ref{eqn:CEFinLM}) follows from the containment $C\subseteq [lm-2k]$ and our choice of $k$.

We now prove (\ref{eqn:CCEFempty}) by  showing $C \cap (C+E)=\varnothing$ and  $C\cap (C+mF)=\varnothing$.  First assume, to get a contradiction, that $C\cap (C+E)\neq \varnothing$. Then there are $c, c'\in C$ and $e\in E$ such that $c=c'+e$.  Representing $c$ and $c'$ as in (\ref{eqn:represent}), we have
\begin{equation}\label{eqn:CCE}
a+qe_0 + m(b+qf_0) = a'+q'e_0 +e + m(b'+q'f_0)
\end{equation}
where $a,a'\in A$, $b,b'\in B$, and $q, q'\in \{0,1\}$.  Thus $a+qe_0$ and $ a'+q'e_0+e$ belong to $[m]$, while $b+qf_0$ and $b'+q'f_0$ belong to $[l]$, so (\ref{eqn:CCE}) and uniqueness of the representation in (\ref{eqn:represent}) implies
\begin{align} \label{eqn:parts1}
a+qe_0&=a'+q'e_0+e,\\
\label{eqn:parts2} b+qf_0&=b'+q'f_0.
\end{align}  Rewriting (\ref{eqn:parts2}), we get $b-b' = (q-q')f_0$.  This implies $q=q'$, since otherwise we get $b-b' = \pm f_0$, contradicting our assumption that $(B-B)\cap F=\varnothing$.  With $q=q'$ equation (\ref{eqn:parts1}) simplifies to $a = a'+e$.  This contradicts our assumption that $A \cap (A+E)=\varnothing$.

To prove that $C\cap (C+mF)=\varnothing$, we argue as above and arrive at the equations $a+qe_0=a'+q'e_0$ and $b+qf_0=b'+q'f_0+f$.  As above, the first equation implies $q=q'$, so $a=a'$, and we get that $b=b'+f$, which again contradicts our assumption that $B\cap (B+F)=\varnothing$.  This completes the proof of (\ref{eqn:CCEFempty}).

Since $C$ is a union of translates of $A$, we may replace $C$ with $C-\min(C)+\min(A)$ to get $A\subseteq C$ and maintain the inclusions $C\subseteq [lm]$ and (\ref{eqn:CEFinLM}). Together with (\ref{eqn:Cdensity}) and (\ref{eqn:CCEFempty}), this shows that $(C,lm)$ witnesses the $2\delta\eta$-nonrecurrence of $E\cup mF$.
\end{proof}

\subsection{Proof of Theorem \ref{thm:Kriz}} Here is our second key ingredient, which we prove in \S\ref{sec:PiecesProofs}.
\begin{lemma}\label{lem:FinitePiecesExist}
 For all $k\in \mathbb N$ and all $\delta\in (0,\frac{1}{2})$, there is a finite set $S\subseteq \mathbb Z$ such that $S$ is $k$-chromatically recurrent and $\delta$-nonrecurrent.
\end{lemma}
We now prove Theorem \ref{thm:Kriz} by combining this lemma with Lemma \ref{lem:2Pieces}.  The proof will use no additional information about the sets provided by Lemma \ref{lem:FinitePiecesExist}.

\begin{proof}[Proof of Theorem \ref{thm:Kriz}] Let $\delta\in (0,\frac{1}{2})$. We will find a chromatically recurrent set $S$ and a set $C\subseteq \mathbb Z$ with $\bar{d}(C)\geq \delta$ such that $(C-C)\cap S=\varnothing$, meaning $S$ is not density recurrent.   To build $S$ and $C$, we will use Lemmas \ref{lem:2Pieces} and \ref{lem:FinitePiecesExist} to  find increasing sequences of sets $S_1\subseteq S_2\subseteq \dots$, $C_1\subseteq C_2\subseteq \dots$, and intervals $[m_1], [m_2],\dots$, with $m_k\to \infty$, so that the following conditions hold for all $k\in \mathbb N$:
\begin{enumerate}
\item[(i)] $S_k$ is $k$-chromatically recurrent,
\item[(ii)] $C_k\subseteq [m_k]$, $C_k+S_k\subseteq [m_k]$, $C_k+S_k+S_k\subseteq [m_k]$, and $|C_k|>\delta m_k$,
 \item[(iii)] $(C_k-C_k)\cap S_k=\varnothing$.
\end{enumerate}
Having constructed these, we let $C:=\bigcup_{k=1}^\infty C_k$ and $S:=\bigcup_{k=1}^\infty S_k$. Then $(C-C)\cap S=\varnothing$; otherwise for some $k\in \mathbb N$ we would have $c-c'=s$ for some $c, c'\in C_k$ and $s\in S_k$.  Item (ii)  implies $\bar{d}(C)\geq \delta$, as $\frac{|C_k\cap [m_k]|}{m_k}>\delta$ for every $k$.  Item (i) implies $S$ is chromatically recurrent, being $k$-chromatically recurrent for every $k\in \mathbb N$.

To find $S_k$ and $C_k$, we first choose $S_1=\{1\}$, $m_1=2$, and $C_1=\{0\}$, so that (i)-(iii) are trivially satisfied with $k=1$.

 To perform the inductive step, we assume the sets $S_k$, $C_k$, and the integer $m_k$ have been constructed to satisfy (i)-(iii), and choose $\delta_k>\delta$ so that $|C_k|>\delta_k m_k$.  Since $\delta_k>\delta$, we may choose $\eta<\frac{1}{2}$ so that $2\delta_k\eta>\delta$.   Lemma \ref{lem:FinitePiecesExist} provides a finite set $S'$ which is $(k+1)$-chromatically recurrent and $\eta$-nonrecurrent.  Apply Lemma \ref{lem:2Pieces} to find $l\geq 2$ and $C_{k+1}\subseteq [l m_k]$ such that $(C_{k+1}, l m_k)$ witnesses the $2\delta_k\eta$-nonrecurrence of $S_k\cup m_kS'$.  Finally, Lemma \ref{lem:DilateChromatic} implies $m_kS'$ is $(k+1)$-chromatically recurrent. Setting $m_{k+1}=l m_k$ and $S_{k+1}=S_k\cup m_kS'$, we get that (i)-(iii) are satisfied with $k+1$ in place of $k$.
\end{proof}

%\subsection{Proof of Corollary \ref{cor:WeakKriz}}
%
%The following lemma will be proved in \S\ref{sec:BohrRecurrence}.  Its proof is based on elementary linear algebra, and does not require chromatic numbers of Kneser graphs.
%
%\begin{lemma}\label{lem:BohrPiecesExist}
%  For all $\delta<\frac{1}{2}$ and $k\in \mathbb N$, there is a finite set $S\subseteq \mathbb Z$ which is $(k,1/k)$-Bohr recurrent and $\delta$-nonrecurrent.
%\end{lemma}
%
%To prove Corollary \ref{cor:WeakKriz} without citing Theorem \ref{thm:Kriz}, the proof of Theorem \ref{thm:Kriz} may be modified by replacing ``$k$-chromatically recurrent'' with ``$(k,1/k)$-Bohr recurrent,'' citing Lemma \ref{lem:BohrPiecesExist} in place of Lemma \ref{lem:FinitePiecesExist}, and citing Lemma \ref{lem:DilateBohrRecurrent} in place of Lemma \ref{lem:DilateChromatic}.

\subsection{Proof of Theorem \ref{thm:KrizInSminusS}}

The following modification of Lemma \ref{lem:FinitePiecesExist} is proved in \S\ref{sec:PiecesProofs}.

\begin{lemma}\label{lem:PiecesInSminusS}
Let $k, m\in \mathbb N$ and $\delta<\frac{1}{2}$. If $E\subseteq \mathbb Z$ is infinite, then there is a finite $\delta$-nonrecurrent set $S\subseteq \mathbb N$ such that $mS\subseteq E-E$ and $mS$ is $k$-chromatically recurrent.
\end{lemma}
Theorem \ref{thm:KrizInSminusS} may be proved by following the proof of Theorem \ref{thm:Kriz} verbatim, with two modifications:

\begin{enumerate}

\item[$\bullet$] Let $S_1=\{t\}$ where $t\in E-E$ satisfies $t>(\frac{1}{2}-\delta)^{-1}$, let $C_1=[t-1]$, and let $m_1=2t$.  Then $|C_1|/m_1 = (t-1)/(2t) = \frac{1}{2}-\frac{1}{2t}>\delta$, and it is easy to check that (i)-(iii) are satisfied with $k=1$.

\item[$\bullet$] Cite Lemma \ref{lem:PiecesInSminusS} in place of Lemma \ref{lem:FinitePiecesExist}, noting in the inductive step that if $S_k \subseteq E - E$, then $S_{k+1}\subseteq E-E$.  \hfill $\square$
\end{enumerate}

%\begin{lemma}\label{lem:SubgroupRecurrent}
%  If $S\subset\mathbb Z$ is $k$-chromatically recurrent and $m\in \mathbb N$, then $S\cap m\mathbb Z$ is $\lfloor k/m\rfloor$-chromatically recurrent.
%\end{lemma}
%
%\begin{proof}
%   Let $r=\lfloor k/m\rfloor$ and $f:\mathbb Z\to \{1,\dots,r\}$ be a coloring.  Form a new coloring $g:\mathbb Z\to \{1,\dots, m\}\times \{1,\dots, r\}$ by defining $g(n)=(n\mod m,f(n))$.  Then $g$ uses fewer than $k$ colors, so there exists $a, b$ with $g(a)=g(b)$ and $b-a\in S$.  Our definition of $g$ implies $b\equiv b\mod m$ whenever $g(a)=g(b)$, so we get that $m$ divides $b-a$.  Thus $b-a\in S\cap m\mathbb Z$.  We also have $f(a)=f(b)$ by the definition of $g$, as desired.
%\end{proof}

\section{Chromatic recurrence and \tpdf{$\delta$}{delta}-nonrecurrence in \texorpdfstring{$\mathbb T^d$}{Td}}\label{sec:BH}

The remainder of this article contains the proofs of Lemmas \ref{lem:FinitePiecesExist} and \ref{lem:PiecesInSminusS}.  Very roughly, the proofs proceed by copying the Hamming balls $H_k(\mb 1)$ from $\mathbb F_2^d$ into $\mathbb Z$, passing through $\mathbb T^d$ as an intermediate step.  To be more precise, we fix some notation.

%
%Definition \ref{def:BH} introduces the sets we use to prove Lemmas \ref{lem:FinitePiecesExist} %, \ref{lem:BohrPiecesExist},
% and \ref{lem:PiecesInSminusS}.  Their relevant properties follow from basic results on uniform distribution, summarized in \S\ref{sec:UD}.

%The key properties here are:
%\begin{enumerate}
%  \item[(i)]  $\mathbb F_2^d$ is isomorphic to the subgroup $G_d:=\{0,\frac{1}{2}\}^d\subseteq \mathbb T^d$.
%  \item[(ii)]  $\mathbb T^d$ is tiled by half-open cubes of the form $\mb t+[0,\frac{1}{2})^d$, where $\mb t\in G_d$.  This means that $\mathbb T^d$ can be written as the disjoint union $\bigcup_{\mb t\in G_d} \mb t + [0,\frac{1}{2})^d$.
%  \item[(iii)] Letting $I_{\varepsilon}:=[\varepsilon, \frac{1}{2}-\varepsilon]$, $\mathbb T^d$ is almost tiled by closed cubes $\mb t + I_{\varepsilon}^d$, meaning the disjoint union $\bigcup_{\mb t\in G_d} \mb t + I_{\varepsilon}^d$ has Haar measure close to $1$.
%  \item[(iv)] If $A\subseteq G_d$, we let $A_{\varepsilon}^{\square}:= A + I_d^\varepsilon$.
%\end{enumerate}

Let $\mathbb R$ denote the real numbers with the usual topology and let $\mathbb T$ denote $\mathbb R/\mathbb Z$ with the quotient topology.  For $x\in \mathbb T$, let $\tilde{x}$ be the unique element of $[0,1)$ such that $\tilde{x}+\mathbb Z = x$, and write $\|x\|$ for $\min_{n\in \mathbb Z}|\tilde{x}-n|$.

When defining subsets of $\mathbb T$, we identify subintervals of $\mathbb R$ with their images in $\mathbb T$ under the quotient map.

Let $G_d:=\{0,1/2\}^d\subseteq \mathbb T^d$, so that $G_d$ is a subgroup isomorphic to $\mathbb F_2^d$.  With the natural identification, $\mb 1\in \mathbb F_2^d$ is identified with $\mbhalf:=(1/2,\dots,1/2)\in \mathbb T^d$.  We will use the notation $H_k$ for Hamming balls around elements of $G_d$, so $H_k(\mbhalf)$ is, by definition, the set of $(x_1,\dots,x_d)\in G_d$ where at most $k$ entries $x_i$ are not equal to $1/2$.

Let $V_{\varepsilon}$ denote the open box $\{\mb x\in \mathbb T^d: \max \|x_j\|<\varepsilon\}$. For $\bm\alpha \in \mathbb T^d$, define
\begin{equation}\label{eqn:BHdef}\tilde{H}(\bm\alpha;k,\varepsilon):=\{n\in \mathbb Z: n\bm\alpha \in H_{k}(\mbhalf) + V_{\varepsilon}\}.
\end{equation}
We call $\tilde{H}(\bm\alpha;k,\varepsilon)$ the \emph{$\varepsilon$-copy of $H_k(\mb 1)$} determined by $\bm\alpha$.

%\begin{remark}
%  A similar object, called a ``Bohr-Hamming ball'' is used \cite{griesmer2020separating}.  The $\varepsilon$-copy of $H_k(\mb 1)$ is smaller than the corresponding Bohr-Hamming ball.
%\end{remark}

The next lemma records the key properties of $\tilde{H}(\bm\alpha;k,\varepsilon)$.  When $\bm\alpha \in \mathbb T^d$ and $E\subseteq \mathbb Z$, we write $E\bm\alpha$ for $\{n\bm\alpha:n\in E\}$.
\begin{lemma}\label{lem:TildeHproperties}\,\begin{enumerate}
\item[(i)] Let $\delta<\frac{1}{2}$ and $k\in \mathbb N$. For all sufficiently large $d\in \mathbb N$ there is an $\varepsilon>0$ such that for all $\bm\alpha\in \mathbb T^d$, the set $\tilde{H}(\bm\alpha;k,\varepsilon)$ is $\delta$-nonrecurrent.

\item[(ii)]  If $\mathbb Z\bm\alpha$ is dense in $\mathbb T^d$ and $\varepsilon>0$, then there is a finite subset of $\tilde{H}(\bm\alpha;2k+1,\varepsilon)$ which is $2k$-chromatically recurrent.

\item[(iii)]  If $E\subseteq \mathbb Z$ is such that $E\bm\alpha$ is dense in $\mathbb T^d$ and $\varepsilon>0$, then there is a finite subset of $\tilde{H}(\bm\alpha;2k+1,\varepsilon)\cap (E-E)$ which is $2k$-chromatically recurrent.
    \end{enumerate}
\end{lemma}
 Parts (ii) and (iii) will be proved in \S\ref{sec:CopyCayley}.  The remainder of this section is dedicated to the proof of Part (i).

Given $\varepsilon>0$, let $I_{\varepsilon}=[\varepsilon,\frac{1}{2}-\varepsilon]\subseteq \mathbb T$, and let $I_{\varepsilon}^d\subseteq \mathbb T^d$ be its $d$-fold cartesian power.  Observe that
\begin{equation}\label{eqn:Tile}
 \text{ the sets } \mb t + I_{\varepsilon}^d,\, \mb t\in G_d, \text{ are mutually disjoint.}
\end{equation}
Write $\mu$ for Haar probability measure on $\mathbb T^d$. Given $A\subseteq G_d$, let $A^{\ssquare}_{\varepsilon}:= A + I_{\varepsilon}^d$, so that (\ref{eqn:Tile}) implies
\begin{equation}\label{eqn:SquareMeasure}
\mu(A^{\ssquare}_{\varepsilon})=|A|\bigl(\tfrac{1}{2}-2\varepsilon\bigr)^d.
\end{equation}

\begin{lemma}\label{lem:BoxIntersections}
  If $\mb t \in G_d$ and $A \subseteq G_d$, then $A^{\ssquare}_{\varepsilon}\cap (A^{\ssquare}_{\varepsilon}+\mb t)=(A \cap (A+\mb t))^{\ssquare}_{\varepsilon}$.
\end{lemma}

\begin{proof}
  We prove only that $A^{\ssquare}_{\varepsilon}\cap (A^{\ssquare}_{\varepsilon}+\mb t)\subseteq (A \cap (A+\mb t))^{\ssquare}_{\varepsilon}$, as the reverse containment is easy to check.

   Assuming $\mb x \in A^{\ssquare}_{\varepsilon}$ and $\mb x\in (A^{\ssquare}_{\varepsilon}+\mb t)$, there are $\mb a, \mb a'\in A$ such that $\mb x \in \mb a + I_{\varepsilon}^d$ and $\mb x\in \mb a'+\mb t + I_{\varepsilon}^d$.  Thus $\mb x\in (\mb a + I_{\varepsilon}^d)\cap (\mb a' + \mb t +I_{\varepsilon}^d)$, and  the mutual disjointness observed in (\ref{eqn:Tile}) implies $\mb a = \mb a'+\mb t$.  It follows that $\mb x\in (A \cap (A+\mb t))^{\ssquare}_{\varepsilon}$, as desired.
\end{proof}
The important consequence of Lemma \ref{lem:BoxIntersections} is that if $A,R\subseteq G_d$ with $A\cap (A+R)=\varnothing$, then $A_{\varepsilon}^{\ssquare}\cap (A_{\varepsilon}^\square + R)=\varnothing$.

A \emph{box} $J$ in $\mathbb T^d$ is a product $J_1\times \cdots \times J_d$ of intervals $J_i\subseteq \mathbb T$; we say that $J$ is open (closed) if each $J_i$ is an open (closed) interval. Note that the sets $A_{\varepsilon}^{\ssquare}$ defined above are finite disjoint unions of closed boxes.

The following standard fact about uniform distribution is proved in \S\ref{sec:UD}.
 \begin{lemma}\label{lem:BoxDensity}
  Let $\bm\alpha\in\mathbb T^d$ be such that $\mathbb Z\bm\alpha$ is dense in $\mathbb T^d$.  If $A\subseteq \mathbb T^d$ is a finite disjoint union of boxes and $C:=\{n\in \mathbb Z: n\bm\alpha\in A\}$ then $d(C)=\mu(A)$, where $\mu$ is Haar probability measure on $\mathbb T^d$.
\end{lemma}

\begin{lemma}\label{lem:LiftNonrecurrence1}
   Let $B$, $U\subseteq \mathbb T^d$ and assume $\mathbb Z\bm\alpha$ is dense in $\mathbb T^d$. If $B$ is a finite disjoint union of boxes with $\mu(B)>\delta$ and $B\cap (B+U) = \varnothing$, then $S:=\{n:n\bm\alpha \in U\}$ is $\delta$-nonrecurrent.
\end{lemma}

\begin{proof}
  Let $C:=\{n:n\bm\alpha \in B\}$.  By Lemma \ref{lem:BoxDensity}, $d(C)= \mu(B)>\delta$.  To see that $C\cap (C+S)=\varnothing$, note that if $n\in C\cap (C+ S)$, we have $n\bm\alpha \in B$ and $n\bm\alpha \in B+U$, meaning $n\bm\alpha \in B\cap (B+U)$, which violates our hypothesis.
\end{proof}

\begin{lemma}\label{lem:LiftNonrecurrence2}
  If $R\subseteq G_d$ is $\delta$-nonrecurrent and $\bm\alpha\in \mathbb T^d$, then for sufficiently small $\varepsilon$, $S:=\{n\in \mathbb Z: n\bm\alpha\in R+V_{\varepsilon}\}$ is $\delta$-nonrecurrent.
\end{lemma}

\begin{proof}
  Let $A\subseteq G_d$ be such that $A\cap (A+R)=\varnothing$ and $|A|>\delta 2^d$.  Choose $\varepsilon'>0$ small enough that $|A|\bigl(\frac{1}{2}-2\varepsilon'\bigr)^d>\delta$.  Let $B:=A^{\ssquare}_{\varepsilon'}$, so that $\mu(B)=|A|\bigl(\frac{1}{2}-2\varepsilon'\bigr)^d>\delta$.  Now $B\cap (B+R)=\varnothing$, by Lemma \ref{lem:BoxIntersections}.  Since $B$ is compact, we also have $B\cap (B+R+V_\varepsilon)=\varnothing$, provided $\varepsilon$ is sufficiently small. Lemma \ref{lem:LiftNonrecurrence1} then implies $S$ is $\delta$-nonrecurrent.
\end{proof}

\begin{proof}[Proof of Lemma \ref{lem:TildeHproperties} Part (i)]
Let $\delta<\frac{1}{2}$ and $k\in \mathbb N$. By Lemma \ref{lem:HNonrecurrent}, choose $d\in \mathbb N$ so that $H_{k}(\mb 1)$ is a $\delta$-nonrecurrent subset of $\mathbb F_2^d$.  Then $H_{k}(\mbhalf)$ is a $\delta$-nonrecurrent subset of $G_d$, and Lemma \ref{lem:LiftNonrecurrence2} provides an $\varepsilon>0$ so that $\{n\in \mathbb Z:n\bm\alpha \in H_k(\mbhalf)+V_{\varepsilon}\}$ is $\delta$-nonrecurrent.  This $\delta$-nonrecurrent set is, by definition, $\tilde{H}(\bm\alpha;k,\varepsilon)$.
\end{proof}

\section{Lifting chromatic recurrence from topological groups}\label{sec:CopyCayley}

Here we show how a homomorphism from $\mathbb Z$ into a topological abelian group $K$ can be used to copy Cayley graphs from $K$ into $\mathbb Z$; we maintain the conventions of \S\ref{sec:CayleyKneser}.
%
%As motivation, note that when $G$ and $H$ are abelian groups, $\rho:G\to H$ is a homomorphism, and $E\subseteq G$ satisfies $\rho(E)=H$, and $S\subseteq H$, then $\Cay(\rho^{-1}(S)\cap (E-E))$ contains an isomorphic copy of $\Cay(S)$ (``isomorphic'' in the sense of graph isomorphism).  This is a special case of Lemma \ref{lem:CopyCayley} below.

For this section we fix a discrete abelian group $G$, a Hausdorff topological abelian group $K$ (not necessarily compact, not necessarily metrizable), and a homomorphism $\rho:G\to K$ with $\rho(G)$ dense in $K$.  We also fix $E\subseteq G$ with $\rho(E)$ dense in $K$.  The only case we will use in this article is $G=\mathbb Z$, $K=\mathbb T^d$, and $\rho(n)=n\bm\alpha$ for some $\bm\alpha\in \mathbb T^d$, so readers can specialize to this setting at will.

\begin{lemma}\label{lem:CopyCayley} Assume $G$, $\rho$, $E$, and $K$ are as specified above. If $U\subseteq K$ is open, then every finite subgraph $\mathcal G$ of $\Cay(U)$ has an isomorphic copy in $\Cay(\rho^{-1}(U)\cap (E-E))$.

Consequently, if $\Cay(U)$ has a finite subgraph with chromatic number $\geq k$, then \\ $\Cay(\rho^{-1}(U)\cap (E-E))$ has chromatic number $\geq k$.
\end{lemma}

\begin{proof}
  To prove the first statement of the lemma it suffices to prove that if $V$ is a finite subset of $K$, then there exist $\{g_v:v\in V\}\subseteq E$ such that for each $v, v'\in V$, we have
  \begin{equation}\label{eqn:vv'}
  v-v'\in U \implies g_v-g_{v'}\in \rho^{-1}(U).
    \end{equation} So let $V$ be a finite subset of $K$.  Let $S:=(V-V)\cap U$, and let $W$ be a neighborhood of $0$ in $K$ so that $S+W\subseteq U$.  Choose a neighborhood $W'$ of $0$ so that $W'-W'\subseteq W$.   For each $v\in V$, choose $g_v\in E$ so that $\rho(g_v)\in v+W'$; this is possible since $\rho(E)$ is dense in $K$.  We now prove (\ref{eqn:vv'}) holds with these $g_v$. Assuming $v-v'\in U$, we have
  \[\rho(g_v)-\rho(g_{v'})\in v+W' -(v'+W') = (v-v')+(W'-W')\subseteq v-v' + W \subseteq U,\]
  so $g_v-g_{v'}\in \rho^{-1}(U)$.  This proves (\ref{eqn:vv'}).

  Since chromatic number is invariant under isomorphism of graphs, the second assertion of the lemma follows immediately from the first.
\end{proof}

%
%\begin{corollary}\label{cor:EminusEchromatic}  With the definitions above, if $U\subseteq K$ and $\Cay(U)$ has a finite subgraph with chromatic number $\geq k$, then $\Cay(\rho^{-1}(U)\cap (E-E))$ has chromatic number $\geq k$.
%\end{corollary}

\begin{proof}[Proof of Lemma \ref{lem:TildeHproperties} Parts (ii) and (iii)]
  Part (ii) follows from Part (iii), so we only prove Part (iii).

  Let $2k\leq d\in \mathbb N$, let $E\subseteq \mathbb Z$, and let $\bm\alpha\in \mathbb T^d$ with $E\bm\alpha$ dense in $\mathbb T^d$.  We use Lemma \ref{lem:CopyCayley} with the open set $U=H_{2k+1}(\mbhalf)+V_{\varepsilon}$ and $\rho(n):=n\bm\alpha$.  Since $H_{2k+1}(\mbhalf)\subseteq U$, $\Cay(U)$ contains the finite subgraph $\Cay(H_{2k+1}(\mbhalf))$ (with vertex set $G_d$). The latter graph is isomorphic to $\Cay(H_{2k+1}(\mb 1))$, and by Lemma \ref{lem:HContainsKneser}, it has chromatic number at least $2k+1$.  Now Lemma \ref{lem:CopyCayley} implies $\Cay(\rho^{-1}(U)\cap (E-E))$ has chromatic number at least $2k+1$.  Examining the definitions, we see that $\rho^{-1}(U)=\tilde{H}(\bm\alpha;2k+1,\varepsilon)$, so $\tilde{H}(\bm\alpha;2k+1,\varepsilon)\cap (E-E)$ is $2k$-chromatically recurrent.
\end{proof}

\section{Proof of Lemmas \ref{lem:FinitePiecesExist} and \ref{lem:PiecesInSminusS}}\label{sec:PiecesProofs}

Lemma \ref{lem:FinitePiecesExist} is an immediate consequence of Parts (i) and (ii) of Lemma \ref{lem:TildeHproperties}. For Lemma \ref{lem:PiecesInSminusS}, we need an elementary result on recurrence properties of dilates.  Given $S\subseteq \mathbb Z$ write $S/m$ for the set $\{n\in \mathbb Z: mn\in S\}$ ($=\{n/m:n\in S\}\cap \mathbb Z$).

\begin{lemma}\label{lem:QuotientNonrecurrence}
  If $m\in \mathbb N$ and $S\subseteq \mathbb Z$ is $\delta$-nonrecurrent, then $S/m$ is also $\delta$-nonrecurrent.
\end{lemma}

\begin{proof}  Let $S\subseteq \mathbb Z$ be $\delta$-nonrecurrent, and
  choose $A\subseteq \mathbb Z$ such that $(A-A)\cap S=\varnothing$ and $\bar{d}(A)>\delta$.  The subadditivity of $\bar{d}$ implies that for some $t\in \{0,\dots,m-1\}$, the set $A_{t}:=A\cap (t+m\mathbb Z)$ satisfies $\bar{d}(A_t)>\delta/m$.  The elements of $A_t-t$ are all divisible by $m$, so $A':=(A_t-t)/m$ satisfies $\bar{d}(A')>\delta$.  Furthermore, $m(A'-A')\subseteq A-A$, so $m(A'-A')$ is disjoint from $S$, meaning $A'-A'$ is disjoint from $S/m$.
\end{proof}

We recall the statement of Lemma \ref{lem:PiecesInSminusS}: if  $k$, $m\in \mathbb N,$ $\delta<\frac{1}{2}$, and $E\subseteq \mathbb Z$ is an infinite set, then there is a finite $\delta$-nonrecurrent set $S\subseteq \mathbb N$ such that $mS$ is $k$-chromatically recurrent and $mS\subseteq E-E$. Our proof  will use the following result from the theory of uniform distribution, which we prove in \S\ref{sec:UD}.
\begin{lemma}\label{lem:InfiniteDense}
  Let $E\subseteq \mathbb Z$ be an infinite set and $d\in \mathbb N$. Then there is an $\bm\alpha \in \mathbb T^d$ such that $E\bm\alpha:=\{n\bm\alpha:n\in E\}$ is topologically dense in $\mathbb T^d$.
\end{lemma}

\begin{proof}[Proof of Lemma \ref{lem:PiecesInSminusS}]
Fix $k, m\in \mathbb N$ and $\delta<\frac{1}{2}$.  Let $E\subseteq \mathbb Z$ be infinite.  Let $E'\subseteq E$ be an infinite subset with $m|(b-a)$ for all $a, b\in E'$ (i.e.~the elements of $E'$ are mutually congruent mod $m$).

By Part (i) of Lemma \ref{lem:TildeHproperties}, choose $d\in \mathbb N$ large enough and $\varepsilon>0$ small enough that for all $\bm\alpha\in \mathbb T^d$,  $\tilde{H}_{\bm\alpha}:=\tilde{H}(\bm\alpha;2k+1,\varepsilon)$ is $\delta$-nonrecurrent.  Using Lemma \ref{lem:InfiniteDense}, we fix $\bm\alpha\in \mathbb T^d$ such that $E'\bm\alpha$ is dense in $\mathbb T^d$.  We then apply Part (iii) of Lemma \ref{lem:TildeHproperties} to find a finite subset $S_0\subseteq \tilde{H}_{\bm\alpha}\cap (E'-E')$ which is $2k$-chromatically recurrent.

Since the elements of $E'$ are mutually congruent mod $m$, we have $S_0\subseteq m\mathbb Z$, and the previous paragraph shows that $S_0$ is $2k$-chromatically recurrent and $\delta$-nonrecurrent.   Let $S:=S_0/m$. By Lemma \ref{lem:QuotientNonrecurrence}, $S$ is also $\delta$-nonrecurrent.  Now $S$ is the desired set: we have $mS=S_0\subseteq E'-E'\subseteq E-E$, $S$ is $\delta$-nonrecurrent, and $mS$ is $2k$-chromatically recurrent.
\end{proof}

\section{Uniform distribution and denseness in \texorpdfstring{$\mathbb T^d$}{Td}}\label{sec:UD}
The material in this section is due to Kronecker and to Weyl; the standard reference is \cite{KuipersNiederreiter}.  We present it here to make the proof of Lemmas \ref{lem:BoxDensity} and \ref{lem:InfiniteDense} self-contained.  For this section we fix $d\in \mathbb N$ and write $\mu$ for Haar probability measure on $\mathbb T^d$.

\subsection{Uniform distribution}

A sequence $\bm\alpha_1,\bm\alpha_2,\dots$ of points in $\mathbb T^d$ is called \emph{uniformly distributed} if
\[\lim_{N\to\infty} \frac{1}{N}\sum_{n=1}^N f(\bm\alpha_n)=\int f\, d\mu\] for every continuous $f:\mathbb T^d\to \mathbb C$.  A straightforward approximation argument shows that a sequence $(\bm\alpha_n)_{n\in \mathbb N}$ is uniformly distributed if and only if for every box $R:=\prod_{j=1}^d [a_j,b_j]$ in $\mathbb T^d$, we have
\[
\lim_{N\to \infty} \frac{|\{n\in \{1,\dots,N\}: \bm\alpha_n\in R\}|}{N} = \mu(R).
\]
Consequently, the terms of a uniformly distributed sequence form a dense subset of $\mathbb T^d$.

The above criterion immediately yields the following lemma.
\begin{lemma}\label{lem:UDDensity} If $A\subseteq \mathbb T^d$ is a finite disjoint union of boxes and $(\bm\alpha_n)_{n\in \mathbb N}$ is uniformly distributed then $C:=\{n:\bm\alpha_n\in A\}$ has $d(C)=\mu(A)$.
\end{lemma}

We write $\mathcal S^1$ for the group of complex numbers of modulus $1$; the group operation on $\mathcal S^1$ is multiplication and the topology is the one inherited from $\mathbb C$.  A \emph{character} of $\mathbb T^d$ is a continuous homomorphism $\chi: \mathbb T^d\to \mathcal S^1$;  the set of characters of $\mathbb T^d$ is denoted by $\widehat{\mathbb T}^d$.  A well known fact from basic harmonic analysis is that every character of $\mathbb T^d$ has the form $\chi(x_1,\dots,x_d):=\exp(2\pi i (n_1x_1+\cdots + n_dx_d))$, where $n_1,\dots,n_d$ are integers.  The \emph{trivial character} satisfies $\chi(\mb x)=1$ for all $\mb x\in \mathbb T^d$.  It is easy to check that if $\chi$ is a nontrivial character and $m\in \mathbb Z\setminus \{0\}$ then $\mb x\mapsto \chi(m \mb x)$ is another nontrivial character.  Furthermore, the characters of $\mathbb T^d$ are mutually orthogonal in $L^2(\mu)$.

The \emph{Weyl criterion} \cite[p.~62]{KuipersNiederreiter} for uniform distribution says that a sequence of points $(\bm\alpha_n)_{n\in \mathbb N}$ in $\mathbb T^d$ is uniformly distributed if and only if
\begin{equation}\label{eqn:Weyl}
\lim_{N\to\infty} \frac{1}{N}\sum_{n=1}^N \chi(\bm\alpha_n)=0 \quad \text{for every nontrivial character } \chi\in \widehat{\mathbb T}^d.
\end{equation}
We now use Weyl criterion to prove Lemma \ref{lem:BoxDensity}, which says that if $(n\bm\alpha)_{n\in \mathbb N}$ is dense in $\mathbb T^d$, $A\subseteq \mathbb T^d$ is a finite disjoint union of boxes, and $C:=\{n\in\mathbb Z:n\bm\alpha \in A\}$, then $d(C)=\mu(A)$.
\begin{proof}[Proof of Lemma \ref{lem:BoxDensity}] By Lemma \ref{lem:UDDensity}, it suffices to prove that if $(n\bm\alpha)_{n\in \mathbb N}$ is dense in $\mathbb T^d$, then (\ref{eqn:Weyl}) is satisfied by $\bm\alpha_n=n\bm\alpha$.  To prove this, we need only show that $\chi(\bm\alpha)\neq 1$, for then we may apply the identity $\chi(n\bm\alpha)=\chi(\bm\alpha)^n$ and simplify the partial sums $S_N:=\sum_{n=1}^N \chi(n\bm\alpha)$ as $\chi(\bm\alpha)\frac{\ \, 1-\chi(\bm\alpha)^N}{1-\chi(\bm\alpha)}$.  Since $|\chi(\bm\alpha)|=1$, the latter quantity is bounded in modulus by a number independent of $N$, whence $\lim_{N\to \infty} \frac{1}{N}S_N=0$.    To see that $\chi(\bm\alpha)\neq 1$, note that the denseness of $(n\bm\alpha)_{n\in \mathbb N}$ in $\mathbb T^d$ implies that $\{\chi(n\bm\alpha):n\in \mathbb N\}$ is dense in the image of $\chi$ (since $\chi$ is continuous), and in particular $\chi(n\bm\alpha)\neq 1$ for some $n$.  This implies $\chi(\bm\alpha)^n\neq 1$, so $\chi(\bm\alpha)\neq 1$.
\end{proof}

\subsection{Denseness}

%The next lemma is the most sophisticated part of our proof of Lemma \ref{lem:InfiniteDense}.
%
%\begin{lemma}\label{lem:PointsOfDensity}
%  If $D\subseteq \mathbb T^d$ is a measurable set and there is a constant $c>0$ such that $\mu(D\cap J)\geq c\mu(J)$ for every box  $J\subseteq \mathbb T^d$ of positive measure, then $\mu(D)=1$.
%\end{lemma}
%
%\begin{proof}  Let $D\subseteq \mathbb T^d$ and $c$ be as in the hypothesis, and let $C=\mathbb T^d\setminus D$.  Assume, to get a contradiction, that $\mu(C)>0$.  By the Lebesgue density theorem (\cite[Section 7.12]{RudinBook}, \cite[Theorem 3.20]{Oxtoby}), there is an box $J\subseteq \mathbb T$ such that $\mu(C\cap J)> (1-c)\mu(J)\geq 0$, meaning $0<\mu(D\cap J)<c\mu(J)$, contradicting the hypothesis.
%\end{proof}

For a set $D\subseteq \mathbb T^d$ and $n\in \mathbb Z$, let $n^{-1}D:=\{x\in \mathbb T^d: nx\in D\}$.  For example, with $D=(0.5,0.6) \subseteq \mathbb T$, we have $3^{-1}(0.5,0.6) = (\tfrac{1}{6},\tfrac{1}{5}) \cup (\tfrac{1}{2}, \tfrac{8}{15}) \cup (\tfrac{5}{6},\tfrac{13}{15})$.

%\begin{lemma}\label{lem:Mixing} Let $I, J\subseteq \mathbb T$ be open intervals and let $E\subseteq \mathbb Z$ be infinite.  Then $E^{-1}I \cap J\neq \varnothing$.
%\end{lemma}
%\begin{proof} Without loss of generality, we may assume that $I$  and $J$ can be identified with intervals $[a,b]$ and $[c,d]$ in $[0,1]$, since $I$ and $J$ can each, in general, be written as a union of two such intervals in $\mathbb T$ overlapping in at most a single point. Under these assumptions, $n^{-1}I$ is a disjoint union of $n$ evenly spaced intervals $L_k = [\frac{a}{n},\frac{b}{n}] + \frac{k}{n}$, $0\leq k \leq n-1$ of length $\mu(I)/n$.  Then $J$ contains each interval $L_k$ where $c<\frac{k}{n}$ and $\frac{b+k}{n}<d$ (and possibly others).  In other words, $nc < k < nd - b$ is sufficient for $J$ to contain $L_k$. So there are at least $\lfloor n(c-d)-b\rfloor$ of the $L_k$ contained in $n^{-1}I\cap J$.  Since $\mu(L_k)=\mu(I)/n$ for each $k$, we have $\mu(n^{-1}I\cap J) \geq \lfloor n(c-d) - b \rfloor \mu(I)/n$, which tends to $ \mu(J)\mu(I)$ as  $n\to \infty$.
%\end{proof}

In the remainder of this section, we use ``open box'' to mean ``nonempty open box.''

\begin{lemma}\label{lem:Mixing}
  Let $I$ and $J$ be open boxes in $\mathbb T^d$.  For all sufficiently large $n$, $n^{-1} I \cap J$ contains an open box.
\end{lemma}

\begin{proof}
  We may assume, without loss of generality, that $I= \prod_{m=1}^d (a_m,b_m)$ and $J = \prod_{m=1}^d (c_m,d_m)$, where $0\leq a_m,b_m,c_m,d_m\leq 1$.  Now $n^{-1}I$ can be written as the union of boxes $L(k_1,\dots,k_d) := L_0 + \frac{1}{n}(k_1,\dots,k_d)$, $0\leq k_m \leq n-1$, $k_m\in \mathbb Z$, where $L_0:=\prod_{m=1}^d (a_m/n,b_m/n)$.  We see that for $n$ sufficiently large, there are $k_m \in [0,n-1]$ such that $c_m \leq  (a_m+ k_m)/n, (b_m+k_m)/n \leq d_m$ for each $m$, meaning $n^{-1}I\cap J$ contains the box $L(k_1,\dots,k_d)$.
\end{proof}

\begin{lemma}\label{lem:HitInterval}
Let $I\subseteq \mathbb T^d$ be an open box and let $E\subseteq \mathbb Z$ be an infinite set.  Then there is a dense open set of $\bm\alpha\in \mathbb T^d$ such that $E\bm\alpha \cap I \neq \varnothing$.
\end{lemma}

\begin{proof}
 We first observe that $\{\bm\alpha \in \mathbb T^d: E\bm\alpha \cap I\neq \varnothing\} = \bigcup_{n\in E} n^{-1}I$.  So it suffices to prove that $\bigcup_{n\in E} n^{-1}I$ contains a dense open set. To see this, note that Lemma \ref{lem:Mixing} implies that for all open boxes $J,$ there is an $n\in E$ such that $n^{-1}I\cap J$ contains an open box.  %Thus, for every open box $J$, $\bigcup_{n\in E} n^{-1}I \cap J$ contains an open box.
 \end{proof}

\begin{proof}[Proof of Lemma \ref{lem:InfiniteDense}] Let $d\in \mathbb N$ and let $E\subseteq \mathbb Z$ be infinite.  Let $\{I^{(k)}:k\in \mathbb N\}$ enumerate the  boxes in $\mathbb T^d$ formed by open intervals with rational endpoints, and for each $k$ let $Q_k:=\{\bm \alpha: E\bm\alpha \cap I^{(k)}\neq \varnothing\}$.  By Lemma \ref{lem:HitInterval}, each $Q_k$ is dense and open in $\mathbb T^d$, so $Q:=\bigcap_{k\in \mathbb N} Q_k$ is nonempty, by the Baire category theorem.  For each box $I^{(k)}$ and each $\bm\alpha\in Q$, we have $E\bm\alpha \cap I^{(k)}\neq \varnothing$, so $E\bm\alpha$ is dense in $\mathbb T^d$.
\end{proof}

\section{Questions and remarks}\label{sec:Questions}
A proof or disproof of the following conjecture would be very interesting in relation to the results of \cite{GriesmerKatznelsonRemarks}.

\begin{conjecture}\label{conj:Hereditary}
  If $S\subseteq \mathbb Z$ is a set of density recurrence, then there is a set $S'\subseteq S$ such that $S'$ is a set of chromatic recurrence and not a set of density recurrence.
\end{conjecture}

Perhaps the proof of Theorem \ref{thm:KrizInSminusS} can be extended to prove Conjecture \ref{conj:Hereditary}.  If so, the key lemma to generalize is Part (iii) of Lemma \ref{lem:TildeHproperties}. Specifically, the required generalization would replace $E-E$ with an arbitrary set of chromatic recurrence, resulting in the following conjecture.  We use the notation $\tilde{H}(\bm\alpha;k,\varepsilon)$ defined in (\ref{eqn:BHdef}).

\begin{conjecture}
  Assume $S\subseteq \mathbb Z$ is a set of chromatic recurrence, $k\in \mathbb N$, and $C>0$.  Then there exist infinitely many $d\in \mathbb N$ and $\bm\alpha \in \mathbb T^d$, such that for all $\varepsilon>0$, the intersection
  \[S\cap \tilde{H}(\bm\alpha; C\sqrt{d},\varepsilon)\]
  is a set of $k$-chromatic recurrence.
\end{conjecture}
See the remarks following Theorem \ref{thm:Lovasz} above to motivate the appearance of $C\sqrt{d}$ here.

\begin{remark}
    There are many constructions of sets with prescribed recurrence properties, such as \cite{BourgainvdC}, \cite{Forrest}, \cite{GrRRPD}, \cite{griesmer2020separating}, \cite{Katznelson} \cite{MountakisvdC}.  These all use one or both of the main ideas of Kriz's construction (i.e.~they use Step I or Step II of the outline in \S\ref{sec:outline} above).

  Among the constructions of sets demonstrating that some recurrence property does not imply density recurrence, \emph{every} such construction is in some way based on properties of Hamming balls (or approximate Hamming balls) in $\mathbb F_p^d$ (or on $\mathbb T^d$).  In other words, there is only one known way\footnote{Sets constructed this way are often referred to as \emph{niveau sets} in the additive combinatorics literature.} to fulfill Step I in our outline.  However, Lemma 2.3 of \cite{griesmer2020separating} provides a very different way to implement Step II.  In contrast to the present article, the implementation of Step II in \cite{griesmer2020separating} is not independent of the form of the pieces from Step I.
\end{remark}

\subsection{Bohr recurrence}  A set $S\subset \mathbb Z$ is a \emph{set of Bohr recurrence} if for all $d\in \mathbb N$, all $\bm\alpha\in \mathbb T^d$, and all $\varepsilon
>0$, there is an $n\in S$ such that $\|n\bm\alpha\|<\varepsilon$.  ``Approximative set'' is the term used for ``set of Bohr recurrence'' in \cite{Ruzsa85}.  It is easy to verify that every set of chromatic recurrence is a set of Bohr recurrence (see \S2 of \cite{Katznelson}, for instance); whether the converse holds is problem made famous by \cite{Katznelson}. This problem is surveyed in \cite{GKR} and discussed in \cite{GriesmerKatznelsonRemarks}.

\begin{remark}\label{rem:AvoidKneser}  As observed by Ruzsa in \cite{Ruzsa85}, the use of Kneser graphs (and hence Theorem \ref{thm:Lovasz}) can be avoided if instead of Theorem \ref{thm:Kriz} one aims to prove the weaker statement ``there is a set of Bohr recurrence which is not a set of density recurrence.''     To do so, one can follow the outline of our proof of Theorem \ref{thm:Kriz}, replacing Step I with ``\textbf{Step I$'$}: find finite sets approximating the property of being Bohr recurrent and not density recurrent.''  Finite subsets of $\tilde{H}(\bm\alpha;k,\varepsilon)$ will suffice for this purpose; elementary proofs of their Bohr recurrence properties appear in \S2.2 of \cite{Katznelson} and in \S8 of \cite{griesmer2020separating}.  Step II can then be followed exactly as in our of Theorem \ref{thm:Kriz}.

However, we do not have a proof of the following corollary of Theorem \ref{thm:KrizInSminusS} that does not prove Theorem \ref{thm:KrizInSminusS} itself.  Hence, we do not have a proof of Corollary \ref{cor:Bohr} avoiding the use of Theorem \ref{thm:Lovasz}.  Such a proof might be useful to the general theory of recurrence.

\begin{corollary}\label{cor:Bohr}
If $E\subset \mathbb Z$ is infinite, then there is an $S\subset E-E$ such that $S$ is a set of Bohr recurrence but not a set of density recurrence.
\end{corollary}
\end{remark}

\subsection*{Acknowledgements}

This article summarizes lectures given by the author in an informal seminar with Nishant Chandgotia, Aryaman Jal, Abhishek Khetan, and Andy Parrish.  I thank them for their interest and encouragement.

An anonymous referee for the \emph{Australasian Journal of Combinatorics} contributed many corrections and improvements.


\frenchspacing
\begin{thebibliography}{10}

\bibitem{Barany}
I.~B\'{a}r\'{a}ny.
\newblock A short proof of {K}neser's conjecture.
\newblock {\em J. Combin. Theory Ser. A}, 25(3):325--326, 1978.

\bibitem{BergelsonERT}
V.~Bergelson.
\newblock Ergodic {R}amsey theory.
\newblock In {\em Logic and combinatorics ({A}rcata, {C}alif., 1985)},
  volume~65 of {\em Contemp. Math.}, pages 63--87. Amer. Math. Soc.,
  Providence, RI, 1987.

\bibitem{BergelsonMcCutcheon}
V.~Bergelson and Randall McCutcheon.
\newblock Recurrence for semigroup actions and a non-commutative {S}chur
  theorem.
\newblock In {\em Topological dynamics and applications ({M}inneapolis, {MN},
  1995)}, volume 215 of {\em Contemp. Math.}, pages 205--222. Amer. Math. Soc.,
  Providence, RI, 1998.

\bibitem{Bogoliouboff}
N.~Bogolio\`uboff.
\newblock Sur quelques propri\'{e}t\'{e}s arithm\'{e}tiques des
  presque-p\'{e}riodes.
\newblock {\em Ann. Chaire Phys. Math. Kiev}, 4:185--205, 1939.

\bibitem{BourgainvdC}
J.~Bourgain.
\newblock Ruzsa's problem on sets of recurrence.
\newblock {\em Israel J. Math.}, 59(2):150--166, 1987.

\bibitem{Folner}
E.~F{\o}lner.
\newblock Note on a generalization of a theorem of {B}ogolio\`uboff.
\newblock {\em Math. Scand.}, 2:224--226, 1954.

\bibitem{ForrestThesis}
A.~H. Forrest.
\newblock {\em Recurrence in dynamical systems: {A} combinatorial approach}.
\newblock ProQuest LLC, Ann Arbor, MI, 1990.
\newblock Thesis (Ph.D.)--The Ohio State University.

\bibitem{Forrest}
A.~H. Forrest.
\newblock The construction of a set of recurrence which is not a set of strong
  recurrence.
\newblock {\em Israel J. Math.}, 76(1-2):215--228, 1991.

\bibitem{Furstenberg77}
H.~Furstenberg.
\newblock Ergodic behavior of diagonal measures and a theorem of
  {S}zemer\'{e}di on arithmetic progressions.
\newblock {\em J. Analyse Math.}, 31:204--256, 1977.

\bibitem{FurstenbergBook}
H.~Furstenberg.
\newblock {\em Recurrence in ergodic theory and combinatorial number theory}.
\newblock Princeton University Press, Princeton, N.J., 1981.
\newblock M. B. Porter Lectures.

\bibitem{GKR}
Daniel Glasscock, Andreas Koutsogiannis, and Florian~K. Richter.
\newblock On {K}atznelson's question for skew-product systems.
\newblock {\em Bull. Amer. Math. Soc. (N.S.)}, 59(4):569--606, 2022.

\bibitem{GreeneKneserShort}
J.~E. Greene.
\newblock A new short proof of {K}neser's conjecture.
\newblock {\em Amer. Math. Monthly}, 109(10):918--920, 2002.

\bibitem{griesmer2020separating}
J.~T. Griesmer.
\newblock Separating {B}ohr denseness from measurable recurrence.
\newblock {\em Discrete Anal.}, Paper No. 9, 20pp.

\bibitem{GrRRPD}
J.~T. Griesmer.
\newblock Recurrence, rigidity, and popular differences.
\newblock {\em Ergodic Theory Dynam. Systems}, 39(5):1299--1316, 2019.

\bibitem{GriesmerKatznelsonRemarks}
J.~T. Griesmer.
\newblock Special cases and equivalent forms of {K}atznelson's problem on
  recurrence.
\newblock {\em Monatsh. Math.}, 200(1):63--79, 2023.

\bibitem{Katznelson}
Y.~Katznelson.
\newblock Chromatic numbers of {C}ayley graphs on {$\mathbb Z$} and recurrence.
\newblock {\em Combinatorica}, 21(2):211--219, 2001.
\newblock Paul Erd\H{o}s and his mathematics (Budapest, 1999).

\bibitem{KuipersNiederreiter}
L.~Kuipers and H.~Niederreiter.
\newblock {\em Uniform distribution of sequences}.
\newblock Wiley-Interscience [John Wiley \& Sons], New York-London-Sydney,
  1974.
\newblock Pure and Applied Mathematics.

\bibitem{Kriz}
I.~K\v{r}\'{\i}\v{z}.
\newblock Large independent sets in shift-invariant graphs: solution of
  \text{Bergelson's} problem.
\newblock {\em Graphs Combin.}, 3(2):145--158, 1987.

\bibitem{Lovasz}
L.~Lov\'{a}sz.
\newblock Kneser's conjecture, chromatic number, and homotopy.
\newblock {\em J. Combin. Theory Ser. A}, 25(3):319--324, 1978.

\bibitem{Matousek}
J.~Matou\v{s}ek.
\newblock {\em Using the {B}orsuk-{U}lam theorem}.
\newblock Universitext. Springer-Verlag, Berlin, 2003.
\newblock Lectures on topological methods in combinatorics and geometry,
  Written in cooperation with Anders Bj\"{o}rner and G\"{u}nter M. Ziegler.

\bibitem{McCutcheon}
R.~McCutcheon.
\newblock Three results in recurrence.
\newblock In {\em Ergodic theory and its connections with harmonic analysis
  ({A}lexandria, 1993)}, volume 205 of {\em London Math. Soc. Lecture Note
  Ser.}, pages 349--358. Cambridge Univ. Press, Cambridge, 1995.

\bibitem{McCutcheonBook}
R.~McCutcheon.
\newblock {\em Elemental methods in ergodic {R}amsey theory}, volume 1722 of
  {\em Lecture Notes in Mathematics}.
\newblock Springer-Verlag, Berlin, 1999.

\bibitem{MountakisvdC}
A.~{Mountakis}.
\newblock {Distinguishing sets of strong recurrence from van der Corput Sets}.
\newblock {\em Israel. J. Math.}, August 2024, \href{https://doi.org/10.1007/s11856-024-2644-7}{ DOI: 10.1007/s11856-024-2644-7}

\bibitem{Ruzsa82}
I.~Z. Ruzsa.
\newblock Uniform distribution, positive trigonometric polynomials and
  difference sets.
\newblock In {\em Seminar on {N}umber {T}heory, 1981/1982}, pages Exp. No. 18,
  18. Univ. Bordeaux I, Talence, 1982.

\bibitem{Ruzsa85}
I.~Z. Ruzsa.
\newblock Difference sets and the {B}ohr topology.
\newblock unpublished manuscript, available at
  \href{https://sites.uml.edu/daniel-glasscock/files/2021/06/Ruzsa_difference_sets_1985.pdf}{https://sites.uml.edu/daniel-glasscock/files/2021/06/Ruzsa\_difference\_sets\_1985.pdf},
  1985.

\bibitem{Sarkozy}
A.~S\'{a}rk\H{o}zy.
\newblock On difference sets of sequences of integers. {I}.
\newblock {\em Acta Math. Acad. Sci. Hungar.}, 31(1-2):125--149, 1978.

\bibitem{Sz}
E.~Szemer\'{e}di.
\newblock On sets of integers containing no {$k$} elements in arithmetic
  progression.
\newblock {\em Acta Arith.}, 27:199--245, 1975.

\bibitem{vdW}
B.~L. van~der Waerden.
\newblock Beweis einer baudetschen vermutung.
\newblock {\em Nieuw Arch. Wisk.}, 15:212--216, 1927.

\bibitem{WeissBook}
B.~Weiss.
\newblock {\em Single orbit dynamics}, volume~95 of {\em CBMS Regional
  Conference Series in Mathematics}.
\newblock American Mathematical Society, Providence, RI, 2000.

\end{thebibliography}
\end{document}